\documentclass{article}
\usepackage{amsthm}
\usepackage{amsfonts}
\usepackage{amssymb}
\usepackage{amsmath}
\usepackage{nccmath}
\usepackage{mathtools}
\usepackage{mathrsfs}
\usepackage{setspace}
\usepackage{graphicx}
\usepackage[export]{adjustbox}
\usepackage{hyperref}
\usepackage[numbers,sort&compress]{natbib}
\usepackage{enumerate}
\usepackage{authblk}
\usepackage{placeins}

\newtheorem{theorem}{Theorem}[section]

\theoremstyle{definition}
\newtheorem{definition}[theorem]{Definition}

\newtheorem{corollary}[theorem]{Corollary}

\theoremstyle{remark}
\newtheorem{remark}[theorem]{Remark}

\numberwithin{equation}{section}

\usepackage{amsmath,amssymb,graphicx} 
\usepackage[margin=1in]{geometry} 
\usepackage{enumerate}
\usepackage{authblk}
\usepackage{placeins}
\usepackage{array}
\usepackage{academicons}
\usepackage{xcolor}
\usepackage{svg}
\usepackage[utf8]{inputenc}
\usepackage{booktabs}
\usepackage{tikz}
\usepackage{placeins}
\providecommand{\keywords}[1]{\textbf{\textit{Keywords:}} #1}
\providecommand{\subjclass}[1]{\textbf{\textit{MSC2020:}} #1}
\begin{document}

\nocite{*} 

\title{ Analysis of the  Differential-Difference Equation  $y(x+1/2)-y(x-1/2) = y'(x)$ }

\author{Hailu Bikila Yadeta \\ email: \href{mailto:haybik@gmail.com}{haybik@gmail.com} }
  \affil{Salale University, College of Natural Sciences, Department of Mathematics, Fiche, Oromia,  Ethiopia}
\date{\today}
\maketitle

\noindent
\begin{abstract}
In this paper we study some solution techniques of differential-difference equation
$$ y'(x)  = y(x + 1/2)- y(x- 1/2),$$
first without an initial condition and then with some initial function $h$ defined on the unit interval $ [-1/2, 1/2]$. We show some sufficient conditions that an initial function $h$ is admissible, i.e., it yields a unique continuous solution on some symmetric interval about $0$.
\end{abstract}

\noindent\keywords{ differential-difference equation, initial value problem, Fourier transform, admissible initial data, characteristic function, characteristic equation }\\
\\

\section{Introduction}
In this paper we study some  type of differential-difference equations. We  make use of the following definition of differential-difference equations.
\begin{definition}
A differential-difference equation is an equation in an unknown function and certain of its derivatives, evaluated
at arguments which differ by any of a fixed number of values. See, for example, \cite{BK}.
\end{definition}
 In other texts, for example, in \cite{KM} a differential-difference equation is defined as a functional differential equation, or a differential equation with  deviating arguments, in which the argument values are discrete. The  general form of differential-difference equation is given by
 \begin{equation}\label{eq:generalform}
     y ^{m} (x) = f(x,  y ^{m_1}(x- \mu_1 (x))), y ^{m_2}(x- \mu_2 (x))),..., y ^{m_k}(x- \mu_k (x))),
 \end{equation}
 where $y(x) \in \mathbb{R}^n$, $ m_1,m_2,...,m_k \geq 0 $, and $ \mu_1 (x),  \mu_2 (x),...,  \mu_k (x) \geq 0 $.
\begin{remark}
  In most textbooks, in place of the scalar variable $x$ that we use here, the scalar variable $t$ which commonly signify time in time-varying process is used. We use $x$ as an independent scalar variable and $y $ as unknown scalar variable that depends on $x$ and its shifts  in this paper.
\end{remark}

 \begin{definition}
   A differential-difference equation (\ref{eq:generalform}) is said to be  \emph{retarded}, \emph{neutral}, or \emph{advanced} according to the quantity of
    $ \max \{ m_1, m_2....,m_k\}$ is \emph{less than},  \emph{is equal to}, or \emph{is greater than} $m$. See \cite{KM}, \cite{BK}.
 \end{definition}
Here are some examples of retarded, neutral, and advanced differential-difference equations.
 \begin{itemize}
   \item $y'(x)= y(x-1)+y(x-2)$, is a retarded differential-difference equation.
   \item  $y'(x)=y'(x-1)+y(x-2)$, is a neutral differential-difference equation.
   \item  $y'(x)= y''(x+2)-y(x-1)$, is  an advanced differential difference equation.
 \end{itemize}
 As listed in the research paper by  E. Yu. Romanenco and  A. N. Sharkovisky (see \cite{DS},\cite{DSA}) \cite{ANS}, one of the three key areas of applications of difference equations with continuous time is in the study of differential-difference equation theory. It is pointed out there that the theory of differential-difference equations, especially differential-difference equations of neutral type, should contain, at least formally, the theory of continuous-time difference equations. In the physical sciences, the differential-difference equations play a vital role in the modeling of complex physical phenomena. The differential-difference models are used in the vibration of particles in lattices, the flow of current in a network, and the pulses in biological chains. For example, see  \cite{BGH},\cite{UIIF}, and the references therein. For abstract theory functional differential equations, solvability and related topics see \cite{AMR}. For Symbolic computation of hyperbolic tangent solutions for nonlinear differential–difference equations see \cite{BGH}. The analysis of some specific differential-difference equations are discussed in some earlier works. For example, see \cite{SJ}.

In this paper, we study the linear differential-difference equation defined on continuous space
\begin{equation}\label{eq:DDE}
  y(x+1/2)-y(x-1/2)-y'(x)=0.
 \end{equation}
 As a motivational introduction to the current problem, we discuss the classical mixing problem appearing in several textbooks of ordinary differential equations. We may regard the current problem as a specific condition of the classical problem.  In the classical mixing problem,  the amount $y(t)$ of a solvent in Kg dissolved in a tanker of solution  at  a time instant $t$ when the solution of a concentration $C_{\text{in}}$ inflows at a volume flow rate $R_{\text{in}}$ in units $m^3 /sec $ and the well-stirred mixture of concentration $C_{\text{out}}$  leaves the tanker at an outflow rate of $R_{\text{out}}$ is calculated. The essential equation of the mixing model is given by
\begin{equation}\label{eq:mixtureproblem}
    \frac{dy}{dt}(t) = C_{\text{in}}(t)R_{\text{in}}(t)-C_{ \text{out}}(t) R_{\text{out}}(t),
\end{equation}
where $R_{\text{in}}$ and $R_{\text{out}}$ are the volume flow rates  of the inflowing and the outflowing fluids expressed in units of $m^3/s $, whereas  $C_{\text{in}} $ and   $C_{ \text{out}}$  are the concentrations  of the inflowing fluid and the outflowing fluids  given in units of $Kg/ m^3 $. So that the quantity $dy/dt $ is given in the units of $Kg/s $.  In most cases  and practical examples, the expression $ C_{\text{in}}(t)R_{\text{in}}(t)-C_{ \text{out}}(t) R_{\text{out}}(t) $ is dependent on the current time instant $t$  only, and not some previous time $t-h$, or some future time $ t+h$ for some positive number $h$. Therefore, the mathematical model given by the equation (\ref{eq:mixtureproblem}) yields an ordinary differential equation. That is
$$ C_{\text{in}}(t) R_{\text{in}}(t)-C_{ \text{out}}(t) R_{\text{out}}(t)= a(t)y(t)+ b(t),   $$
for some functions $a$ and $b$. See, for example, \cite{NSS}, \cite{KH}.
Another possible model is  the nonlinear model given in the form
$$ \frac{dy}{dt}(t)= Q(t,y(t)) $$
where the net mass flow rate $Q$ is given in units of $Kg/s$.This model is a nonlinear differential equation.
For the current problem, assume that the rate of change of the dissolved solvent in the tank, $\frac{dy}{dt}$, at a time instant $t$, is determined by the difference between the total amount of solvent available in the tank at some future time $t+1/2 $ and  the total amount of solvent dissolved in the tank at a previous time $t-1/2 $.
 As a result, the model equation in this case becomes $y'(t)=y(t+1/2)-y(t-1/2)$. This model differers from both the classical mixing problem where the problem is linear and the nonlinear model where the mass flow rate depends on the current instant $t$ the quantity $y(t)$ at the instant. It is not a differential equation model but a differential-difference equation model with two shifts  $t+1/2 $ and $t-1/2 $ in the independent variable.

 Assume that for some change in the quantities  $R_{\text{in}}(t)$,  $R_{\text{out}}(t)$, $C_{\text{in}}(t)$, and $C_{\text{out}}(t)$, the simultaneous equation
\begin{equation}\label{eq:decomposition}
    \begin{cases}
      y'(t)= C_{\text{in}}(t)R_{\text{in}}(t)-C_{ \text{out}}(t) R_{\text{out}}(t) \\
      y(t+1/2)-y(t-1/2) = C_{\text{in}}(t)R_{\text{in}}(t)-C_{ \text{out}}(t) R_{\text{out}}(t),
    \end{cases}
  \end{equation}
  has a solution. This is exactly the solution of the differential-difference equation (\ref{eq:DDE}).
   The first of these equations has solutions that can be found easily as a linear differential equation, and the second one is a linear difference equation with continuous argument. What is important is that the two equations in (\ref{eq:decomposition}) should have coinciding solutions, which is also a solution to the differential-difference equation.

We show that the solution to the equation (\ref{eq:DDE})  consists of an infinite number of solutions. Among the possible sets of solutions is the set of all polynomials of degree less than or equal to $2$, $\{y(x) =ax^2+bx+ c, a,b,c \in \mathbb{ R}  \}$.   We find some class of solutions to the equation, including analytic solutions that can be represented in Taylor's series. The initial value  problem for equation \ref{eq:DDE}) has a  unique solution for appropriate initial function $y_0$, where an initial function $y_0$ is given and set to be defined on the unit interval $[-1/2, 1/2]$. We shall also show that the interval of existence of the solution and the smoothness of the solution  depend on the smoothness of the initial function $y_0$.

Another  use of  this differential-difference equation (\ref{eq:DDE}) is to find a curve $y(x)=f(x)$ defined on $ \mathbb{R} $, with the property that the slope of  the chord connecting the two points $ (x-1/2 , y(x-1/2) )$ and $ (x+1/2 , y(x+1/2) )$  on the curve is equal to the slope of the tangent line at the point $ (x, y(x))$. So, the slope of the secant over any unit interval in the domain is equal to the slope of the tangent line at the midpoint of the unit interval.
 \begin{center}
     \includegraphics[width=0.5 \textwidth]{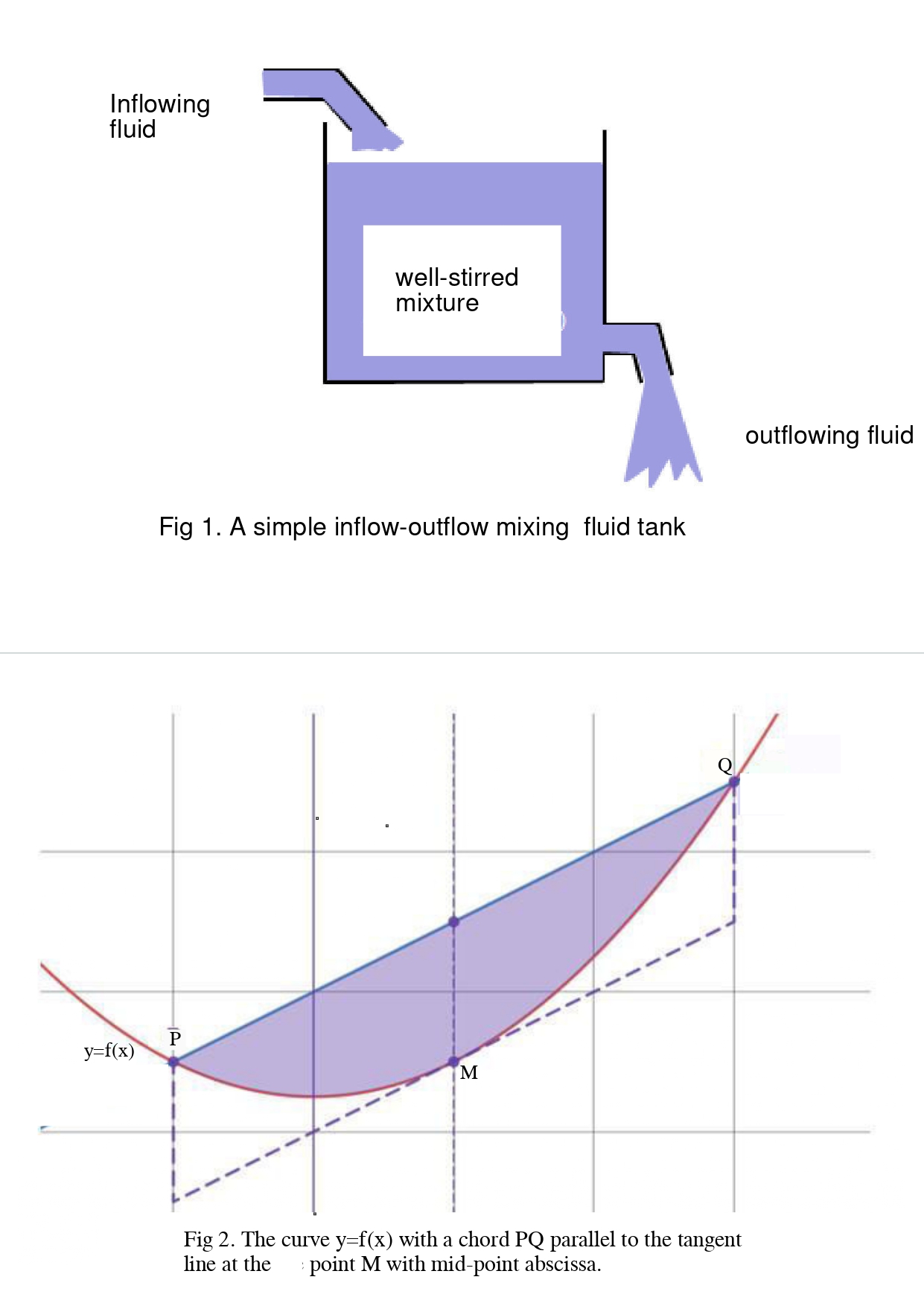}
      \end{center}

\section{ The Differential-Difference Equation  $y(x+1/2)-y(x-1/2) = y'(x) $}

\subsection{Definitions of some Operators and their Relations}
For each $ h \in \mathbb{R} $, we define the shift operator $ E ^h $, and the identity operator $ I $ as
$$ E ^h y(x) := y(x+h),\quad  I y(x) := y(x) .$$
We write $ E^h $ as $E$ rather  than $ E^{1}$ when $h=1$. We define $ E^0 $ as $ I $.
We define the forward difference operator $\Delta $ and the backward difference operator $ \nabla $ as follows:
$$ \Delta y(x):=  (E-I)y(x)= y(x+1) -y(x),\quad  \nabla y(x)= (I-E^{-1})y(x)= y(x)-y(x-1). $$
For $h> 0$, the central difference operator $ \delta ^h $ is defined as
$$ \delta ^h y(x):= \frac{y(x+h)-y(x-h)}{2h} . $$
Finally, we denote  the central difference operator by $L$, which is a special case of $ \delta ^h $, where $h=1/2 $, and the differential operator by $D$,  as follows:
\begin{equation}\label{eq:differencialanddifference}
 Ly(x):=y(x+1/2)-y(x-1/2),\quad  Dy(x):= \frac{d}{dx}y(x)=y'(x).
\end{equation}

\begin{remark}
   We observe the following relations among difference operators:
   $$ E ^{-1/2} L = \nabla, \quad E ^{1/2} L = \triangle, \quad L^2 = \nabla \triangle .$$
   Therefore, the operator $ L $ is the geometric mean of the forward operator $ \triangle $ and the backward operator $\nabla $. Also observe that
   $$ L^2= E-2I+ E^{-1}= (E-I)-(I- E^{-1}) =\triangle - \nabla .$$
 \end{remark}
For detailed discussion of difference operators and their calculus see, for example, \cite{CHR}, \cite{MT},\cite{LB}.
\begin{theorem}\label{eq:paritytheorem}
  The operators $ L $ and $ D $ are \emph{parity-changing operators}. That is, under these operators the image of  even(odd)function is odd (even) function.
\end{theorem}
\begin{proof}
  If $f$ is an even function, i.e, $ f(-x)=f(x)$. Then
  \begin{align*}
  Lf(-x) & =f(-x+1/2)-f(-x-1/2) \\
    & = f(x-1/2)-f(x+1/2) \\
    & -Lf(x).
  \end{align*}
  Therefore $ Lf $ is an odd function. If $g$ is odd function, i.e., $ g(-x)=-g(x)$, then
  \begin{align*}
  Lg(-x) & =g(-x+1/2)-g(-x-1/2) \\
    & = -g(x-1/2)+ g(x+1/2) \\
    &= Lg(x).
  \end{align*}
  Therefore, $ Lg $ is an even function.
\end{proof}

\begin{corollary}
Let $n \in \mathbb{N} $. Define
\begin{equation}\label{eq:Ssubnofx}
  S_n(x):= L x^n= (x+1/2)^n -(x-1/2)^n.
\end{equation}
If $ n $ is odd, then the function $ S_n $ is even; if $ n$ is even, then $ S_n $ is odd.
\end{corollary}

 \begin{proof}
 If $n$ is even integer, then $y(x)= x^n $ is even function;  if $n$ is odd integer, then $y(x)= x^n $ is odd function. As the result of  Theorem \ref{eq:paritytheorem}, the corollary follows.
 \end{proof}
 The values of the polynomials $S_m(x)$ for $ 1\leq x \leq 10 $are shown in Table 1.
\FloatBarrier
\begin{table}
  \centering
  \caption{ Table for $S_m(x)$ for $ 1 \leq m \leq 10 $.}
  \begin{tabular}{||c c ||}
 \hline
 $ m $ & $ S_m(x):= L x^{m}= \sum_{k=0}^{m} \binom{m}{k}x^k\left[ (\frac{1}{2}) ^{m-k}- (-\frac{1}{2}) ^{m-k}\right]$ \\ [0.5ex]
 \hline\hline
 m=1  &  $1$  \\
\hline
m=2  &  $2x$ \\
\hline
m=3  & $ \frac{1}{4}  +3 x^2$  \\
\hline
m=4  & $ x + 4x^3 $    \\
\hline
m=5  & $  \frac{1}{16} +\frac{5}{2} x^2 +  5x^4 $    \\
\hline
m=6  &  $ \frac{3}{8}x +5x^3 +6x^5 $ \\
\hline
m=7  &  $\frac{1}{64}  + \frac{21}{16}  x^2+ \frac{35}{4}x^4+ +7x^6 $ \\
\hline
m=8  & $ \frac{1}{8}x + \frac{7}{2} x^3 + 14 x^5 +8x^7 $  \\
\hline
m=9  &  $ \frac{1}{128} +  \frac{9}{16} x^2 + \frac{63}{8} x^4 + \frac{21}{2} x^6 +9x^8 $ \\
\hline
m=10  &  $  \frac{10}{256} x + \frac{15}{8} x^3 +  \frac{63}{2} x^5 + 30 x^7 +10x^9 $  \\ [1ex]
 \hline
\end{tabular}
\end{table}
\FloatBarrier

\begin{theorem}
  Let $ n \in \mathbb{N} $. Then
  \begin{equation}\label{eq:ssubtwonx}
     S_{2n}(x) = \sum_{k=1}^{n} 2 ^{2k-2n}\binom{2n}{2k-1}x^{2k-1},
  \end{equation}
  \begin{equation}\label{eq:ssubtwonminusonex}
      S_{2n-1}(x) = \sum_{k=0}^{n} 2 ^{2k-2n}\binom{2n-1}{2k}x^{2k}.
  \end{equation}
 \end{theorem}
\begin{proof}
  We prove only (\ref{eq:ssubtwonx}), and the proof of (\ref{eq:ssubtwonminusonex}) is similar to that of (\ref{eq:ssubtwonx}).
  \begin{align*}\label{eq:ssubtwonproof}
    s_{2n}(x) & = \left(x+\frac{1}{2}\right)^{2n}- \left(x -\frac{1}{2}\right)^{2n}  \\
     & =\sum_{r=0}^{2n}   \left( \frac{1}{2}\right)^{2n-r}\binom{2n}{r} x^r - \sum_{r=0}^{2n} \binom{2n}{r}  \left( -\frac{1}{2}\right)^{2n-r} x^r \\
     &= \sum_{r=0}^{2n}  \left[ \left( \frac{1}{2}\right)^{2n-r}-\left( -\frac{1}{2}\right)^{2n-r}\right]\binom{2n}{r} x^r \\
     & = \sum_{k=1}^{n} 2 ^{2k-2n}  \binom{2n}{2k-1} x^{2k - 1}.
  \end{align*}
  This is due to the fact that the power  $ x^r $ vanishes for even $ r $. Re-indexing sum $ r= 0 $  to $ 2n $ as a sum $ k= 1 $  to $ n $ yields ( \ref{eq:ssubtwonx}).
  \end{proof}

\begin{corollary}
$ S_n(x)$ is a polynomial of degree $ n-1 $  for each $ n\in \mathbb{N} $. Furthermore, if we write $S_n(x) $ in the expanded form, we get
$$ S_n(x) = \sum_{i=0}^{n-1}S_{n,k}x^k,  $$
 with the numerical coefficients $ S_{n,k} \geq 0 $.
\end{corollary}

In the subsections that follow, we study the techniques and properties of solutions of the scalar differential-difference equation  given in (\ref{eq:DDE}). That can also be written in operator form as
$$ (L-D)y(x)=0. $$

\subsection{Solutions by Taylor Series Method}

 In this and subsequent subsections, we look  some methods for solving the differential-difference equation (\ref{eq:DDE}), first without an initial value and then with some initial  function $ h $ defined on the symmetric unit interval $ [-1/2, 1/2 ]$.  The Taylor series expansion of the solution $y$ is ne method of solving (\ref{eq:DDE}). This method is helpful for finding an analytic solution to the differential-difference equation.As we can see, the method necessitates an infinite number of numerical coefficients in the power series of the analytic solution $y(x)$. However, here we find only some analytic solutions, while the complete task is equivalent to solving a system of infinite linear equations with an infinite number of unknowns. Let us assume a solution of the differential-difference equation (\ref{eq:DDE}) that may be written in an infinite power series of the form
\begin{equation}\label{eq:analyticsolution}
   y(x)= \sum_{n=0}^{\infty} a_n x^n .
\end{equation}

Then
\begin{equation}\label{eq:Dofyofx}
   Dy(x) = y'(x) = \sum_{n=0}^{\infty} a_{n+1}(n+1)x^n,
\end{equation}

and
\begin{equation}\label{eq:Lofyofx}
  Ly(x)= \sum_{n=0}^{\infty} a_n S_n(x),
\end{equation}
where $ S_n(x) $ is as defined in (\ref{eq:Ssubnofx}).

\begin{theorem}\label{eq:undeterminedcoefficents}
  Assume that $y$,  whose Taylor series is given by (\ref{eq:analyticsolution}), is an analytic solution of the differential-difference equation  (\ref{eq:differencialanddifference}). Then we have the following two homogeneous systems of infinite linear equations in infinite unknowns  $ a_3, a_4, a_5, ...$

\begin{equation}\label{eq:eventriangularsystem}
 \sum_{n=2 + k}^{\infty}   2^{2k-2n} \binom{2n-1}{2k}a_{2n-1}=0,\quad k=0,1,2,...
\end{equation}

\begin{equation}\label{eq:oddtriangularsystem}
 \sum_{n=1+ k}^{\infty}    2^{2k-2n} \binom{2n-1}{2k}a_{2n}=0,\quad k=1,2,...
\end{equation}
\end{theorem}
\begin{proof}
  \begin{align*}
  Dy(x)=D( a_0+ a_1 x + a_2 x^2 + a_3x^3 + a_4x^4+...) &= Ly(x)= L( a_0+ a_1 x+ a_2 x^2 + a_3x^3 + a_4x^4+...)   \\
  \Leftrightarrow 0+ a_1+ 2a_2 x+ 3a_3x^2+ 4a_4x^3+... &= 0 + a_1+ 2a_2x + a_3S_3(x)+a_4S_4(x)+...  \\
  \Leftrightarrow 3a_3x^2+ 4a_4x^3+... & = a_3S_3(x)+a_4S_4(x)+...  \\
  \Leftrightarrow 0 &=  a_3S'_3(x)+a_4S'_4(x)+...
  \end{align*}
  where $ S'_{2n}(x)= S_{2n}(x)-2n x ^{2n-1},\,  n \in \mathbb{N} $ and $ S'_{2n-1}(x)= S_{2n}(x)-2n x ^{2n-2},\, n \geq 2,\, n \in \mathbb{N}$.

\begin{equation}\label{eq:doublesumforevens}
 \sum_{n=2}^{\infty}a_{2n-1}S'_{2n-1}(x)= \sum_{n=2}^{\infty}a_{2n-1}\left(  \sum_{k=0}^{n-1}  2^{2k-2n} \binom{2n-1}{2k} x^{2k}\right)=0
\end{equation}
\begin{equation}\label{eq:doublesumforodds}
 \sum_{n=2}^{\infty}a_{2n}S'_{2n}(x)= \sum_{n=2}^{\infty}a_{2n}\left(  \sum_{k=1}^{n-1}  2^{2k-2n} \binom{2n}{2k-1} x^{2k-1}\right)=0
\end{equation}
From (\ref{eq:doublesumforevens}), equating the sum of all coefficients of the even power $ x^{2k}$ for each $k=0,1,2,3,...$, we get an infinite triangular system of homogeneous equations ( \ref{eq:eventriangularsystem}).

The second triangular system of infinite homogeneous equations (\ref{eq:oddtriangularsystem}) is obtained from (\ref{eq:doublesumforodds}), by equating the sum of all coefficients of the odd power $ x^{2k-1}$ for each $k=1,2,3,...$. This completes the proof.
\end{proof}
In Theorem \ref{eq:undeterminedcoefficents}, the two  systems of infinite linear equations (\ref{eq:eventriangularsystem}) and (\ref{eq:oddtriangularsystem}) in infinite unknowns  $ a_3, a_4, a_5...$ induced by the Taylors series method, the coefficients $a_0,\, a_1,\, a_2 $ appearing in the solution $ y(x)= \sum_{i=0}^{\infty} a_i x ^i $ are free and arbitrary (are not involved in the  systems of infinite linear equations). The infinite systems being homogeneous, setting all the coefficients $ a_3, a_4, a_5...$ equal to zero, we shall obtain the set of solutions that comprise any polynomial in $x$ of second degree or less. Hence the following theorem arises.
\begin{theorem}
   Any polynomial of degree less than or equal to $2$, i.e., $y(x)= a_0 + a_1 x + a_2x^2, \, a_0, a_1, a_2 \in \mathbb{R} $ is a solution of (\ref{eq:differencialanddifference}).
\end{theorem}
\begin{proof}
  Direct substitution yields the desired result.
\end{proof}

\begin{remark}
Observe that  $ L 1= D 1=0,\, Lx =Dx=1,\, L x^2= Dx^2 =2x $, whereas $ \triangle x^2 = 2x+1 \neq 2x = D x^2 $,
 and $ \nabla x^2 = 2x-1 \neq 2x = D x^2 $. The space $ \mathcal{P}_2 $ of all polynomials of degree less than or equal to 2 is contained in the null space of the operator $ L-D $.
 \end{remark}

\subsection{Complex Solutions}
For the differential-difference equation (\ref{eq:DDE}), applying the Fourier transform both sides we get
\begin{equation}\label{eq:Fouriertransform}
   i\xi \hat{y}(\xi) =( e^{i\frac{\xi}{2}} - e^{-i\frac{\xi}{2}})\hat{y}(\xi)= 2i\sin(\xi/2)\hat{y}(\xi),
\end{equation}
where $\hat{y}(\xi)=\int_{-\infty}^{\infty} e^{- i x \xi} y(x)dx $.
From (\ref{eq:Fouriertransform}) we need to find the solutions in $ \mathbb{C} $  of the transcendental equation
\begin{equation}\label{eq:Transendental}
 \xi/2 = \sin(\xi/2).
\end{equation}
\begin{theorem}
 If $ z= a+bi , \, a, b \in \mathbb{R} $ is a solution of the equation(\ref{eq:Transendental}), then
\begin{equation}\label{eq:exp}
   y(x)= e^{izx}
\end{equation}
is a complex solution of the differential-difference equation (\ref{eq:DDE}).
\end{theorem}
\begin{proof}
Let $ y(x)= e^{izx}$, where $ z $ is solution of (\ref{eq:Transendental}). Then
  \begin{align*}
  Ly = y(x+1/2)-y(x-1/2)&= e^{iz(x+1/2)}- e^{iz(x-1/2)} \\
   &=e^{izx}(e^{iz/2}-e^{-iz/2})=e^{izx}2i\sin(z/2)\\
   &=e^{izx} 2i(z/2)=ize^{izx}= Dy(x).
\end{align*}
\end{proof}

\begin{theorem}
  $ z= a + bi , \, a, b \in \mathbb{R} $, is the solution of the transcendental equation (\ref{eq:Transendental}) if and only if $ (x,y)=(a,b)$ is the solution to the system of equations
  \begin{equation}\label{eq:simultaneous}
     \begin{cases}
        & x/2= \sin(x/2)\cosh(y/2),  \\
        &  y/2= \cos(x/2)\sinh (y/2).
     \end{cases}
  \end{equation}
  \end{theorem}

\begin{proof}
 A complex number $z= a + bi$ is a solution of (\ref{eq:Transendental})
  \begin{align*}
      \Leftrightarrow a/2 + i b/2 &= \sin(a/2+ib/2) \\
            & =\sin(a/2)\cos(ib/2)+\cos(a/2)\sin(ib/2)  \\
       &=\sin(a/2)\cosh(b/2)+i\cos(a/2)\sinh(b/2)\\
     \Leftrightarrow   a/2= \sin(a/2)\cosh(b/2) & \quad \text{and}\quad b/2= \cos(a/2)\sinh(b/2).
  \end{align*}
So $ (x,y)=(a,b)$ satisfies the system of equations (\ref{eq:simultaneous}).
\end{proof}
\begin{theorem}\label{eq:abtheorem}
  Let $z= a+bi, a,b \in \mathbb{R}$ is any solution of the transcendental equation (\ref{eq:Transendental}). Then the real part $ y(x)= \Re (e^{izx}) = e^{-bx}\cos(ax) $ and the imaginary $ y(x)= \Im (e^{izx})=e^{-bx}\sin(ax) $ are  then solutions to the differential-difference equation (\ref{eq:differencialanddifference}).
\end{theorem}
\begin{proof}
  Let $y(x)= e^{-bx}\cos ax $. Then $ Dy(x)= -be^{-bx}\cos ax -ae^{-bx}\sin ax $.
  \begin{align*}
    Ly(x)= & y(x+1/2)-y(x-1/2) \\
     & = e^{-b(x+\frac{1}{2})}\cos a(x+1/2)-e^{-b(x-\frac{1}{2})}\cos a(x-1/2) \\
     & = e^{-bx}\left[e^{-\frac{b}{2}}(\cos ax \cos(a/2)-\sin ax \sin(a/2))- e^{\frac{b}{2}}(\cos ax \cos(a/2)+\sin ax \sin(a/2)) \right]\\
     &= e^{-bx} \left[ -2\cos ax \cos(a/2)\sinh(b/2)-\sin ax \sin(a/2)\cosh(b/2)\right]\\
     & =-be^{-bx}\cos ax -ae^{-bx}\sin ax \\
     &= Dy(x).
  \end{align*}
  The verification for $ y(x)= e^{-bx}\sin ax $ is similar.
\end{proof}
As to the existence of  a solution $ (x,y)=(a,b)$ of the system of equations (\ref{eq:simultaneous}), we have the following solutions  of (\ref{eq:simultaneous})  calculated by WOLFRAM ALPHA  \copyright,
\begin{align*}
  & a=-3.75626 \times  10^{-8} \quad  \text{ and}  \quad  b=2.25842 \times 10^{-9 }, \\
  & a=0  \quad \text{ and}\quad  b=-4.79706 \times 10^{-8},  \\
  & a=0 \quad \text{ and} \quad b=0, \\
  & a=0 \quad \text{ and} \quad  b=4.00874  \times 10^{-8}, \\
  & a= 2.10292\times 10^{-8} \quad    \text{ and} \quad  b = 4.04457 \times 10^{-9}.
\end{align*}
Thus, using Theorem \ref{eq:abtheorem} we have additional solutions to the differential-difference equation (\ref{eq:DDE}) other than the ones that we have discussed in the previous section.

\subsection{Integral Equation form of the Differential-Difference Equation}
\begin{theorem}
  The differential-difference equation (\ref{eq:DDE}) can be written as an integral equation
  $$ y(x)= y(0)-\int_{-\frac{1}{2}}^{\frac{1}{2}}y(s)ds + \int_{-\infty}^{\infty} \alpha(x-s)y(s)ds,  $$
  where $ \alpha(x)= \chi_{[-1/2, 1/2]}(x) $  is the characteristic function of the unit interval $[-1/2, 1/2] $.
\end{theorem}

\begin{proof}
  Note that
$$ Ly(x)=y(x+1/2)-y(x-1/2)= \frac{d}{dx}\int_{x-\frac{1}{2}}^{x+\frac{1}{2}} y(s)ds $$
provided that $y \in  C[ x-1/2, x+1/2 ]$ for every $ x \in \mathbb{R} $. Therefore,
$$ Dy(x)-Ly(x)=0 \Leftrightarrow \frac{d}{dx}\left(y(x)-\int_{x-\frac{1}{2}}^{x+\frac{1}{2}} y(s)ds\right)=0. $$
Thus, the expression $y(x)-\int_{x-\frac{1}{2}}^{x+\frac{1}{2}} y(s)ds = c,\,  x \in \mathbb{R}  $, where $c$ is some constant. Setting $ x=0 $ yields the constant $ c = y(0)-\int_{-\frac{1}{2}}^{\frac{1}{2}}y(s)ds $. Hence, the  equivalent integral equation representation for the differential-difference equation is
$$ y(x)= y(0)-\int_{-\frac{1}{2}}^{\frac{1}{2}}y(s)ds + \int_{x-\frac{1}{2}}^{x+\frac{1}{2}} y(s)ds. $$
We further note that
$$ \int_{x-\frac{1}{2}}^{x+\frac{1}{2}} y(s)ds = \int_{-\infty}^{\infty} \alpha(x-s)y(s)=\int_{-\infty}^{\infty} \alpha(s)y(x-s)ds := (\alpha \ast y)(x), $$
where $\ast $ is the convolution. So, we write the differential-difference equation (\ref{eq:DDE}) as an integral equation
$$ y(x)= y(0)-\int_{-\frac{1}{2}}^{\frac{1}{2}}y(s)ds + \int_{-\infty}^{\infty} \alpha(x-s)y(s)ds. $$
This completes the proof of  the theorem.
\end{proof}

\subsection{The Initial Value Problem for the Differential-Difference Equation}

\begin{definition}
  Let $ I $  be some open interval in $ \mathbb{R} $. For integers $ k \geq 0 $, we denote by $ C^k( I )$  the space of functions which are $k$ times continuously differentiable in $ I  $. In particular, by $C^0(I )$ or just $C(I)$, the space of all continuous functions defined in $ I $. Also $ C^\infty(I):= \bigcap \limits_{k\geq 0} C^k(I)$. However, if $ I $ is a closed interval like $[-1/2, 1/2]$, by $h \in C^k( I )$ we mean that $h \in C^k( J )$, where $ I \subset J $ and $ J $ is some open interval in  $ \mathbb{R} $.
  \end{definition}

\begin{theorem}
Let $ k \in \mathbb{N}$. Consider the differential-difference equation \ref{eq:DDE} with additional conditions
\begin{equation}\label{eq:ckinitialvalue}
  \begin{cases}
     & y(x)= h(x), \, x \in [-1/2, 1/2], \, h \in C^k [-1/2, 1/2],\\
     & h^{(i)}(0)= h^{(i-1)}(1/2)- h^{(i-1)}(-1/2),\quad i=1,2....,k.
  \end{cases}
\end{equation}
where $ h^{(i)}$ is the $i$-th order derivative and $  h^{(0)}$ is considered as $ h $. Then there exist a unique solution $y \in C[-k/2, k/2]$  that satisfies the differential-difference equation (\ref{eq:DDE}) whenever $-k/2 \leq x-1/2 < x< x+1/2 \leq k/2 $.
\end{theorem}

\begin{proof}
 We use induction over $k$. If $k=1$, the only point $x$ such that $-k/2 \leq x-1/2 < x< x+1/2 \leq k/2 $ exists is $x=0$. The differential-difference equation (\ref{eq:DDE}) is satisfied at this point by the given initial condition. $ y(x)= h(x), \, x \in [-1/2, 1/2]$ is the solution is. Now we consider the case where $k=2$.  Let $ x \in (1/2, 1]$. Then $ x -1/2 \in (0, 1/2]$ and $ x -1 \in (- 1/2, 0]$. Hence
\begin{equation}\label{eq:firstrightext}
  y(x)= y'(x-1/2) + y(x-1)= h'(x-1/2) + h(x-1),\quad  x \in (1/2, 1].
\end{equation}
 The left hand limit of $y$ at $x = 1/2$ is calculated as follows from the given initial function $h$:
  \begin{equation}\label{eq:leftlimitathalf}
    \lim \limits_{x\rightarrow \frac{1}{2}-} y(x)=   \lim \limits_{x\rightarrow \frac{1}{2}-} h(x)= h(1/2).
  \end{equation}
  By (\ref{eq:firstrightext}), We have

\begin{equation}\label{eq:rightlimitathalf}
  \lim \limits_{x\rightarrow \frac{1}{2}+} y(x)= \lim \limits_{x\rightarrow \frac{1}{2}+} h'(x-1/2) + h(x-1) =h'(0)+h(-1/2).
  \end{equation}

 By using (\ref{eq:leftlimitathalf}) and (\ref{eq:rightlimitathalf}), using the condition given in (\ref{eq:ckinitialvalue}) as a bridge, we get
  \begin{equation}\label{eq:continuityathalf}
  \lim \limits_{x\rightarrow \frac{1}{2}-} y(x)= h(1/2) = h'(0)+ h(-1/2)= \lim \limits_{x\rightarrow \frac{1}{2}+} y(x).
  \end{equation}
Equation  (\ref{eq:continuityathalf}) establishes continuity of $y$ at $ x=1/2 $. Using the given condition on $ h $, we calculate the right derivative at  $ x=1/2 $ as

\begin{align}\label{eq:rightderivativeathalf}
 \lim \limits_{x\rightarrow  \frac{1}{2}+}  \frac{y(x)-y(1/2)}{x-1/2}& = \lim \limits_{x\rightarrow \frac{1}{2}+}  \frac{h'(x-1/2)+h(x-1)-h(1/2)}{x-1/2} \nonumber \\
   & = \lim \limits_{x\rightarrow \frac{1}{2}+} h''(x-1/2)+h'(x-1) \nonumber \\
   &= h''(0)+h'(-1/2)=h'(1/2).
\end{align}
The left hand derivative at $ x=1/2 $ is
\begin{equation}\label{eq:leftderivativeathalf}
 \lim \limits_{x\rightarrow  \frac{1}{2}-}  \frac{y(x)-y(1/2)}{x-1/2}= \lim \limits_{x\rightarrow  \frac{1}{2}-}  \frac{h(x)-h(1/2)}{x-1/2}= h'(1/2).
\end{equation}

 Therefore, \ref{eq:rightderivativeathalf} and \ref{eq:leftderivativeathalf} imply that $ y $ is differentiable at $ x=1/2 $. Because $h \in c^2[-1/2, 1/2]$ and by (\ref{eq:firstrightext}), $y$ is left continuous at $ x=1 $.  Let $ x \in (-1, -1/2]$. Then $ x +1/2 \in (-1/2, 0]$ and $ x +1 \in (0, 1/2]$. We have

 \begin{equation}\label{eq:firstleftext}
  y(x)=  y(x+1)-y'(x+1/2)= h'(x-1/2)+h(x-1),\quad  x \in (-1, -1/2].
\end{equation}
 We can show that $y$ is differentiable at $ x = -1/2 $ and right continuous at $ x = -1$ by using (\ref{eq:firstleftext}) and arguments that are similar to those of  $ x=1/2 $ and $x=1$.  This proves that for the initial function $h$ that meets the conditions in (\ref{eq:ckinitialvalue}) for $k=2$, we have a unique solution $ y \in C[-1, 1]$ that solves the differential-difference equation (\ref{eq:DDE}). Assume that the hypothesis holds true for any arbitrary value of $ k \in \mathbb{N} $. Then we have to prove that the hypothesis works for $k+1$ as well. Consider the differential-difference equation (\ref{eq:DDE}) with the additional conditions
  \begin{equation}\label{eq:ckplusoneinitialvalue}
  \begin{cases}
     & y(x)= h(x), \, x \in [-1/2, 1/2], \, h \in C^{k+1} [-1/2, 1/2],\\
     & h^{(i)}(0)= h^{(i-1)}(1/2)- h^{(i-1)}(-1/2),\quad i=1,2....,k, k+1.
  \end{cases}
\end{equation}
  Let us denote  $h'(x):=g(x),\, -1\leq x \leq 1/2 $.
   Now let us take $k$ of the $k+1$  conditions on $h$
   $$ h^{(i)}(0)= h^{(i-1)}(1/2)- h^{(i-1)}(-1/2),\, i=2...,k, k+1, $$
 that is equivalent to
 $$ g^{(i)}(0)= g^{(i-1)}(1/2)- g^{(i-1)}(-1/2),\quad i=1,...,k .$$
 With these $k$ conditions  let us denote by $\tilde{y} $ that satisfy the following conditions
   \begin{equation}\label{eq:modifiedckinitialvalue}
  \begin{cases}
     & \tilde{y}'(x)= \tilde{y}(x+1/2)-\tilde{y}(x-1/2),\\
     & \tilde{y}(x)= g(x), \, x \in [-1/2, 1/2], \, g \in C^{k} [-1/2, 1/2],\\
     & g^{(i)}(0)= g^{(i-1)}(1/2)- g^{(i-1)}(-1/2),\quad i=1,2....,k.
  \end{cases}
\end{equation}
Then by the induction assumption, there exists a unique solution $  \tilde{y}_k \in C( -k/2, k/2)$ of the differential-difference equation (\ref{eq:DDE}).  However the solution  $ \tilde{y}_k $ is a linear combination of shifts of $ g, g',..., g^{k-1}$. Since $ g \in C^{k} [-1/2, 1/2]$, $ \tilde{y}_k \in C^{1} [-1/2, 1/2] $. Therefore by left and right extension

\begin{equation}\label{eq:extendedmodified}
   y_{k+1}(x)=\begin{cases}
    &  \tilde{y}_k'(x+1/2)+y_k(x+1),\quad [-(k+1)/2,  -k/2) \\
     & \tilde{y}_k(x), \quad -k/2\leq x \leq k/2 \\
     &  \tilde{y}_k'(x-1/2)+y_k(x-1)  \quad (k/2 , (k+1)/2  ].
  \end{cases}
\end{equation}
Now we have to prove that $y_{k+1}  $ is differentiable at $x =\pm k/2 $ and left continuous at $ x=(k+1)/2 $ and right continuous at $  x = -(k+1)/2 $ .
    For continuity at $x = k/2 $

    \begin{equation}\label{eq:leftlimitatkover2}
    \lim \limits_{x\rightarrow \frac{k}{2}-} y_{k+1}(x)=   \lim \limits_{x\rightarrow \frac{1}{2}-} y_k(x)= y_k(k/2).
  \end{equation}

  \begin{align}\label{eq:rightlimitatkover2}
   \lim \limits_{x\rightarrow \frac{k}{2}+} y_{k+1}(x)&= \lim \limits_{x\rightarrow \frac{k}{2}+} y_k'(x-1/2) + y_k(x-1) =   y_k'(k/2-1/2)+ y_k(k/2-1) \nonumber \\
     & = y_k(k/2)-y_k(k/2-1)+y_k(k/2-1)= y_k(k/2)
  \end{align}

  Hence,  by(\ref{eq:leftlimitatkover2}) and(\ref{eq:rightlimitatkover2}), continuity at  $x = k/2 $ is proved. That of $x = -k/2 $ is proved similarly.
  since $ y_k \in C^1 [-k/2,k/2]$ the left hand side derivative of $y_[k+1] $ at  $x = k/2 $ is $y_k'(k/2)$.
  \begin{align}\label{eq:rightderivativeatkover2}
 \lim \limits_{x\rightarrow  \frac{k}{2}+}  \frac{y_{k+1}(x)-y_{k+1}(k/2)}{x-k/2}& = \lim \limits_{x\rightarrow \frac{k}{2}+}  \frac{y_k'(x-1/2)+y_k(x-1)-y_k(k/2)}{x-k/2} \nonumber \\
   & = \lim \limits_{x\rightarrow \frac{1}{2}+} y_k''(x-1/2)+y_k'(x-1) \nonumber \\
   &= y_k'(k/2)- y_k'(k/2-1)+  y_k'(k/2-1 = y_k'(k/2).
\end{align}
Since $y_k \in C^1 $ is $y_{k+1}$ is continuous on $[-(k+1)/2,(k+1)/2] $.
\end{proof}

\begin{theorem}\label{eq:cinfinityIVP}
Consider the differential-difference equation (\ref{eq:DDE}) with additional conditions
\begin{equation}\label{eq:cinfinityinitialvalue}
  \begin{cases}
     & y(x)= h(x), \, x \in [-1/2, 1/2], \, h \in C^\infty [-1/2, 1/2],\\
     & h^{(i)}(0)= h^{(i-1)}(1/2)- h^{(i-1)}(-1/2),\quad i \in \mathbb{N}.
  \end{cases}
\end{equation}
Then there exist a unique solution $y \in C ^ \infty (\mathbb{R })$  of the differential-difference equation.
 \end{theorem}

\begin{proof}
  For mathematical necessity let us consider the restrictions the initial function $h$ as
  \begin{equation}\label{eq:initaldatabreak}
    \begin{cases}
      h(x)|_{(-1/2, 0]} & := y_{-1}(x)  \\
      h(x)|_{(0, 1/2]} & := y_{0}(x).
    \end{cases}
  \end{equation}
   By applying the operator $ E^{-1/2}$ to the differential-difference equation in \ref{eq:DDE} and rearranging, we get
  \begin{equation}\label{eq:forwardextension}
    y(x)= y(x-1)+ y'(x-1/2),\, x \in \mathbb{R}.
  \end{equation}
  Let $x \in (1/2, 1]$. Then  $x-1 \in (-1/2, 0]$, and  $x-1/2 \in (0, 1/2]$. Accordingly, by (\ref{eq:initaldatabreak}) and (\ref{eq:forwardextension})

  \begin{equation}\label{eq:firstforward}
    y(x)  = y_{-1}(x-1)+ y_0'(x-1/2):=y_1(x),\quad x \in (1/2, 1].
  \end{equation}
Thus we have calculated the value $y $ on a new interval  $(1/2, 1] $. Let us denote by $y_{n}$ the value of  $y$ obtained on the interval $ ( n/2, (n+1)/2], n \in \mathbb{N} $. Then we have the recurrence relation
$$  y_{n}(x)  = y'_{n-1}(x-1/2)+ y_{n-2}(x-1)= E^{-1/2}Dy_{n-1}(x)+E^{-1}y_{n-2}(x), $$

which yields a difference equation on continuous space and with operator coefficients
\begin{equation}\label{eq:forwardrecurrence}
  y_{n}(x)- E^{-1/2}Dy_{n-1}(x)+ E^{-1}y_{n-2}(x)=0.
\end{equation}
The characteristic equation of the difference equation (\ref{eq:forwardrecurrence}) is given by
\begin{equation}\label{eq:backwardopratorcharactersticeq}
  \lambda^2  - \lambda E^{-1/2}D + E^{-1} =0,
\end{equation}
 and the roots of the characteristic equation are given by
$$ \lambda = \lambda_1 = E^{-1/2} \Phi (D),\quad \lambda = \lambda_2 = E^{-1/2} \Psi (D),$$
where
\begin{equation}\label{eq:PhiandPsiodD}
  \Phi (D)= \frac{D + \sqrt{D^2+4}}{2} ,\quad  \Psi (D)= \frac{D - \sqrt{D^2+4}}{2}.
\end{equation}
 For arbitrary function $ A $ and $ B $, the general solution of (\ref{eq:forwardrecurrence}) takes the form
\begin{equation}\label{eq:forwardsolform}
  y_{n}(x)= E ^{-n/2} \Phi ^n( D ) A(x)+  E ^{-n/2} \Psi ^n( D )B(x).
\end{equation}
The specific values of $ A $ and $ B $ for the current initial value problem are determined by the given initial functions $y_{-1}$ and $ y_{0}$  as
 \begin{equation}\label{eq:AandB}
     A(x)= \frac{E^{1/2} \Psi ^{-1}y_0(x)-y_{-1}(x) }{\sqrt{ D^2 +4 }},\quad B(x)= \frac{  y_{-1}(x)- E^{1/2}\Phi ^{-1}y_0(x) }{\sqrt{ D^2 +4 }}.
 \end{equation}
Replacing the values of $ A(x)$ and $ B(x)$ written in (\ref{eq:AandB}) into (\ref{eq:forwardsolform}) and then rearranging yields
\begin{equation}\label{eq:ysubn}
    y_n(x) = \frac{ E ^{-(n+1)/2}}{\sqrt{D^2+4}}  [ \Phi ^n(D) -\Psi ^n (D)] y_{-1}(x) - \frac{ E ^{-n/2}}{\sqrt{D^2+4}}  [ \Phi ^{n+1}(D) -\Psi ^{n+1}(D)] y_{0}(x).
\end{equation}
 By applying the operator $ E^{1/2}$ to the differential-difference equation in \ref{eq:DDE} and rearranging, we get

\begin{equation}\label{eq:backwardextension}
    y(x)= y(x+1)- y'(x+1/2),\, x \in \mathbb{R}.
  \end{equation}
  Let $ x \in (-1, -1/2 ]$. Then  $x+1 \in (0, 1/2]$, and  $x + 1/2 \in (-1/2, 0]$. Accordingly, by (\ref{eq:initaldatabreak}) and (\ref{eq:backwardextension})

  \begin{equation}\label{eq:firstbackward}
     y(x)  = y_{0}(x+1)- y_{-1}'(x+1/2):= y_{-2}(x),\quad x \in (-1, -1/2].
  \end{equation}
Thus we could calculate $ y $ on a new interval $(-1, -1/2] $. Let $y_{-n}$  be the calculated  value of  $y$  defined on the interval $ ( -n/2, (1-n)/2], n \in \mathbb{N} $. We get the general recurrence relation
 $$ y_{-n}(x)  = y_{2-n}(x+1)- y'_{1-n}(x+1/2)= Ey_{2-n}(x) - D E^{1/2}y_{1-n}(x),  $$
which yields a second order difference equation on continuous space and with operator coefficients
\begin{equation}\label{eq:backwardrecurrence}
  Ey_{2-n}(x)- E^{1/2} D y_{1-n}(x )- y_{-n}(x)=0.
\end{equation}
  The  characteristic equation for (\ref{eq:backwardrecurrence}) is given by
\begin{equation}\label{eq:backwardopratorcharactersticeq}
   E \lambda^2 -  E^{1/2}D \lambda -1 =0.
\end{equation}
The  roots  of the characteristic equation (\ref{eq:backwardopratorcharactersticeq}) are given by
$$ \lambda = \lambda_1 = E^{1/2} \Phi (D),\quad \lambda = \lambda_2 = E^{ 1/2} \Psi (D), $$
where $  \Phi (D)$ and $\Psi (D)$ are as defined in (\ref{eq:PhiandPsiodD}). In a similar procedure that led us to (\ref{eq:ysubn}), we obtain
\begin{equation}\label{eq:ysubminusn}
    y_{-n}(x) = \frac{ E ^{(n-1)/2}}{\sqrt{D^2+4}}  [ \Phi ^n(D) -\Psi ^n (D)] y_{-1}(x) + \frac{ E ^{n/2}}{\sqrt{D^2+4}}  [ \Phi ^{n-1}(D) -\Psi ^{n-1}(D)] y_{0}(x).
\end{equation}
Combining (\ref{eq:ysubn}) and (\ref{eq:ysubminusn}), we get the solution
\begin{equation}\label{eq:solution}
  y(x)= \sum_{n= -\infty}^{\infty} y_{n}(x) \chi_{(n/2, (1+n)/2] }(x).
\end{equation}
Now we proceed to the proof of uniqueness of the solution given in (\ref{eq:solution}).  Suppose that $ y $, $ \tilde{y} $ are solutions of initial value problem for differential-difference equation. From the given initial condition, $ y(x) = \tilde{y}(x) $ on the interval $ [-1/2,1/2]$.  Consequently, $y_i=  \tilde{y}_i,\, i=-1,0 $ by (\ref{eq:initaldatabreak}). We follow by induction  to prove  $y_i = \tilde{y}_i,\, i \in \{-1,0\} \cup \mathbb{N} $. Suppose that $y_i=  \tilde{y}_i $ for some $i=k,k+1,\, k=-1,0,1,...$, where $y_i$ is the part of the solution defined on the interval $(i/2, (i+1)/2], \, i=-1,0,...$. Then  by forward extension relation (\ref{eq:forwardextension}), we get $y_{k+2}=  \tilde{y}_{k+2}$. A similar argument follows for the backward extension. So  $ y_i(x)=  \tilde{y}_i(x)$ on $ \mathbb{R} $. This completes the  proof of the theorem.
\end{proof}

\begin{remark}
  In the proof of Theorem \ref{eq:cinfinityIVP}, we are not interested in the operational definition of the operators $\Phi $  and $\Psi $  which involve some square roots. However, for every $ n \in \mathbb{N} $, $ \frac{\Phi^n  -\Psi^n  }{\sqrt{D^2+4}}$ is a polynomial (radical free) in $ D $ which has a usual definition, whereas $\Phi^0= \Psi^0= I $ is the identity operator so that $\Phi ^0- \Psi^0 $ is the zero map. Indeed,
  \begin{align*}
    \Phi^n(D)  -\Psi^n(D)  &=  \left( \frac{D + \sqrt{D^2+4}}{2}\right)^n-\left( \frac{D- \sqrt{D^2+4}}{2}\right)^n \\
     & = \frac{1}{2^n} \sum_{s=0}^{n}\binom{n}{s}D^{n-s} ( D^2+4)^{s/2}- \frac{1}{2^n} \sum_{s=0}^{n} (-1)^s \binom{n}{s}D^{n-s} ( D^2+4)^{s/2} \\
     & = \frac{1}{2^n} \sum_{s=0}^{n} (1+(-1)^s) \binom{n}{s} D^{n-s} (D^2+4)^{s/2}\\
     & = \frac{1}{2^{n-1}}\sum_{k=0}^{\lfloor \frac{n-1}{2}\rfloor}\binom{n}{2k+1} D^{n-2k-1}(D^2+4)^k \sqrt{D^2+4}
  \end{align*}
    The terms with even index $ s $ vanish. Because in that case $ 1+(-1)^{1+s}= 0$. Therefore,

    \begin{equation}\label{eq:phndandpsindexpanded}
          \frac{\Phi^n(D)  -\Psi^n(D)}{\sqrt{D^2+4}}= \frac{1}{2^{n-1}}\sum_{k=0}^{\lfloor \frac{n-1}{2}\rfloor}\binom{n}{2k+1} D^{n-2k-1}(D^2+4)^k .
    \end{equation}

\end{remark}

\begin{remark}
  The condition that the initial function $ h \in C ^ \infty[-1/2, 1/2]$  alone does not guarantee the  existence solution $y $. That is why  we include additional condition $ h^{i}(0) = h^{i-1}(1/2)- h^{i-1} (-1/2), i=1,2,...,k $. We may define initial function $h$  satisfying this additional condition as \emph{admissible initial data}. For example, if  $ h(x)= e^x,\,  x \in  [-1/2, 1/2] $, then $y(x):= y_1(x)= e^x(e^{-1}+e^{-1/2}),\, x \in (1/2, 1]$, showing that the solution is discontinuous at $x=1/2$. Hence $h(x)=e^x$ is not \emph{admissible initial data}. On the other hand, if we select the initial function $ h(x)= x^2,\,  x \in  [-1/2, 1/2] $, then $y=(x):= y_1(x)= x^2, \,  x \in  [-1/2, 1/2]  $. In fact, in this case $y(x)=x^2, x \in \mathbb{R}  $, which is a smooth function is an \emph{admissible initial data} as well.
\end{remark}
\section{Discussions of the results}
This paper discusses the differential-difference equations (\ref{eq:DDE}) as an alternative mathematical model of the classical mixing problem of fluid flow. Because  the classical model is usually linear differential equation, the solution techniques are more obvious. Here, we established the existence of solutions for the differential-difference equation using different methods. In this model, we assumed that $y(x)$ to be the amount of a solute dissolved in the unit volume at an instant $x$. It is customary to use $t$  for time, and $x$ for spatial variable, for physical interpretation. Thus we may think $x$ to be time. In the subsection where we applied Fourier transform method, we assumed $x$ as a spacial variable on $\mathbb{R}= (-\infty \infty)$.

 Furthermore, the current problem is interpreted as  finding  all plane curves $y =f(x)$ where the slope of  the chord  connecting two points $(x-1/2, y(x-1/2)$ and $(x+1/2, y(x+1/2)$ on the curve is always equal to the slope of the tangent line at the point  with mid-point abscissa. Because the $[ x-1/2, x+1/2 ]$ is a unit interval, the difference $y(x+2/2)-y(x-1/2) $ is just the chord's slope.

 As regards the solutions of the differential-difference equation (\ref{eq:DDE}), by inspection and consideration of the properties of centred difference and the derivative, we may observe that constant functions, linear functions, and quadratic functions are solutions of the differential-difference equation (\ref{eq:DDE}). These solutions fall into the class of analytic functions. Considering Taylor's series expansion, we could reach  this conclusion. As the first few terms, including the constant term, are removed from the equation, the terms involving $x$, and $x^2$  can have  arbitrary coefficients.

 However, there   may  exist  further solutions that are the outcome of the infinite system of linear equations that are established by Taylor's series method.  We have shown, by Fourier transform methods,  the existence of some other solutions.

 However it remains a question whether the solutions that were obtained by Fourier transform methods are the possible solutions that would be obtained from Taylor's series.  The other component of the current paper  is the initial value problem for the differential-difference equation (\ref{eq:DDE}).   The interval of existence of solutions for the initial value problem for the differential difference equation \ref{eq:DDE} depends on the order of smoothness of the initial data. In general,  we  get a unique solution on the existence interval, $[-k/2, k/2]$ for the initial data $y_0 \in C ^k[-1/2, 1/2)$. The continuation of this solution requires a set of $k$ additional conditions given in (\ref{eq:ckplusoneinitialvalue}) are satisfied by $y_0$. Therefore smoothness of the initial function $y_0$ alone does not guarantee the continuity of the solution.

\section{Conclusions and Possible Future Works}
In this paper we have discussed some kind of linear differential-difference equation on continuous space, its solution techniques, including its initial value problem.  The explicit closed form solution of the initial value problem is formulated. There may be a wider class  of differential-difference equation,$Dy(x)= g(y,Ly )$   which we may name \emph{differential-difference equation nonlinear in the difference part}, and $ L y(x)= f( y, Dy(x))$ which we may name \emph{differential-difference equation nonlinear in the differential part}. This types of problems may be studied without or with some given initial conditions. Some real life application in science and engineering may be incorporated.

\section*{Conflict of Interests}
The author declare that there is no conflict of interests regarding the publication of this paper.

\section*{Acknowledgment}
The author is thankful to the anonymous reviewers for their constructive and valuable suggestions.

\section*{Funding} This Research work is not funded by any institution or person.

\end{document}